\documentclass{article}
\usepackage[utf8]{inputenc}
\usepackage{graphicx,epsfig}
\usepackage{amsmath,amsthm,mathrsfs} 
\usepackage{amssymb} 
\usepackage{amsfonts}
\usepackage{mathtools}
\usepackage{epstopdf}
\usepackage{ifthen}
\usepackage{bbm}
\graphicspath{{eps/}}
\newtheorem{remark}{Remark} 
\newtheorem{definition}{Definition}
\newtheorem{proposition}{Proposition}
\newtheorem{theorem}{Theorem}
\newtheorem{corollary}{Corollary}
\newtheorem{lemma}{Lemma}

\usepackage{lineno,hyperref}
\modulolinenumbers[5]

\title{ Duality Mapping for Schatten Matrix Norms\thanks{This work was funded by the Swiss National Science Foundation under  Grant 200020\_184646.}}

\author{Shayan Aziznejad, and Michael Unser  
\thanks{\'Ecole polytechnique f\'ed\'erale de Lausanne, Lausanne, Switzerland. E-mail: shayan.aziznejad@epfl.ch, michael.unser@epfl.ch}}
\begin{document}
\maketitle
 
\begin{abstract}
In this paper, we fully characterize the duality mapping over the space of matrices that are equipped with Schatten norms. Our approach is based on  the analysis   of the saturation of the H\"older inequality for Schatten norms. We prove in our main result that, for $p\in (1,\infty)$, the duality mapping over the space of real-valued matrices with Schatten-$p$ norm is a continuous and single-valued function and provide an explicit form for its computation. For the special case $p=1$, the mapping is set-valued; by adding a rank constraint, we show that it can be reduced to a Borel-measurable single-valued function for which we also provide a closed-form expression.
\end{abstract}


\section{Introduction}
\label{Sec:intro}
In linear algebra and matrix analysis, Schatten norms are a family of spectral  matrix norms that are defined via the singular-value decomposition \cite{bhatia2013matrix}. They have  appeared in many applications    such as image reconstruction \cite{lefki2013Poisson,lefki2013HS}, image denoising \cite{xie2016weighted}, and tensor decomposition \cite{gao2020robust}, to name a few.

Generally, the Schatten-$p$ norm of a matrix  is the $\ell_p$ norm of  its singular values. The family  contains some well-known matrix norms: The Frobenius and the spectral (operator) norms are special cases in the family, with $p=2$ and $p=\infty$, respectively.  The case $p=1$ (trace or nuclear norm) is of particular interest for applications as it can be used to recover low-rank matrices \cite{davenport2016overview}. This is the current paradigm in matrix completion, where  the goal is to recover an unknown matrix given some of its entries \cite{candes2009exact}.   Prominent examples of applications that can be reduced to low-rank matrix-recovery problems are phase retrieval \cite{candes2015phase}, sensor-array processing \cite{davies2012rank},  system identification \cite{fazel2013hankel}, and index coding \cite{asadi2017fast,esfahanizadeh2014matrix}. 

In addition to their many applications in data science, Schatten norms have been extensively studied from a theoretical point of view. Various inequalities concerning Schatten norms have been proven \cite{kittaneh1985inequalities,kittaneh1987inequalities2,kittaneh1986inequalities3,kittaneh1986inequalities4,kittaneh1987inequalities5,bourin2006matrix,Hirzallah2010Schatten,moslehian2011schatten,conde2016norm};
 sharp bounds for commutators  in Schatten spaces have been given \cite{wenzel2010impressions,cheng2015schatten}; moreover,  facial structure \cite{so1990facial}, Fr\'echet differentiablity \cite{potapov2014frechet}, and various other aspects \cite{kittaneh1989continuity,bhatia2000cartesian}  have been studied already. 
 
 Our objective in this paper is to investigate the duality mapping in spaces of matrices that are equipped with Schatten norms. The duality mapping is a powerful tool to understand the topological structure of Banach spaces \cite{beurling1962theorem,cioranescu2012geometry}. It has been used to derive powerful characterizations of the solution of variational problems in function spaces \cite{de1976best,unser2019unifying} and also to determine generalized linear inverse operators \cite{liu2007best}.  Here, we prove that the duality mapping over Schatten-$p$ spaces with $p\in (1,+\infty)$ is a single-valued and continuous function which, in fact, highlights the strict convexity of these spaces. For the special case $p=1$, the mapping is set-valued. However, we prove that, by adding a rank constraint, it reduces to a single-valued Borel-measurable function. In both cases, we also derive closed-form expressions that allow one to compute them explicitly.

The paper is organized as follows: In Section \ref{Sec:Prelim}, we present relevant mathematical tools and concepts that are used in this paper. We study the duality mapping of Schatten spaces and propose our main result in Section \ref{Sec:DualMap}. We provide further discussions regarding the introduced mappings in Section \ref{Sec:discuss}.
\section{Preliminaries}
\label{Sec:Prelim}
\subsection{Dual Norms, H\"older Inequality, and Duality Mapping}
\label{Sec:Holder}
Let $V$ be a finite-dimensional vector space that is equipped with an inner-product $\langle \cdot,\cdot \rangle:V \times V \rightarrow \mathbb{R}$ and let $\|\cdot\|_{X}:V\rightarrow \mathbb{R}_{\geq 0}$ be an arbitrary norm on $V$. We then denote by $X$  the space $V$ equipped with $\|\cdot\|_X$. Clearly, $X$ is a Banach space, because all finite-dimensional normed spaces are complete. The dual norm of $X$, denoted by $\|\cdot\|_{X'}: V\rightarrow \mathbb{R}_{\geq 0}$, is defined as 
\begin{equation}\label{Eq:DualNorm}
\|{\bf v}\|_{X'} = \sup_{{\bf u} \in V\backslash \{\boldsymbol{0}\} } \frac{\langle {\bf v} , {\bf u} \rangle}{\|{\bf u}\|_{X}},
\end{equation}
for any ${\bf v} \in V$. Following this definition, one would directly obtain the generic duality bound 
\begin{equation}\label{Eq:DualityBound}
\langle {\bf v}, {\bf u} \rangle  \leq \|{\bf v}\|_{X'} \|{\bf u}\|_{X},
\end{equation}
for any ${\bf v}, {\bf u} \in V$. Saturation of Inequality \eqref{Eq:DualityBound} is the key concept of dual conjugates that is formulated in the following definition.
\begin{definition}\label{Def:DualConj}
Let $V$ be a finite-dimensional vector space and let $(\|\cdot\|_X,\|\cdot\|_{X'})$ be a pair of dual norms that are defined over $V$. The pair $({\bf u},{\bf v}) \in V\times V$ is said to be a  $(X,X')$-conjugate, if 
\begin{itemize}
\item $\langle {\bf v}, {\bf u}\rangle = \|{\bf v}\|_{X'} \|{\bf u}\|_{X}$,
\item $\|{\bf v}\|_{X'} = \|{\bf u}\|_{X}$.
\end{itemize}
For any ${\bf u}\in V$, the set of all elements $ {\bf v} \in V$ such that $({\bf u},{\bf v})$ forms an  $(X,X')$-conjugate is denoted by $\mathcal{J}_{X}({\bf u})\subseteq V$. We refer to the set-valued mapping $\mathcal{J}_{X}: V \rightarrow 2^V$ as the duality mapping. If, for all ${\bf u}\in V$, the set  $\mathcal{J}_{X}({\bf u})$ is a singleton, then we indicate the duality mapping for the $X$-norm via the single-valued  function ${\rm J}_X:V\rightarrow V$ with $\mathcal{J}_{X}({\bf u}) = \{{\rm J}_X({\bf u})\}$. 
\end{definition}
It is worth mentioning that,  for any ${\bf u}\in V$, the set  $\mathcal{J}_{X}({\bf u})$  is nonempty .  In 
fact, the closed ball $B= \{ {\bf v}\in V : \|{\bf v} \|_{X'}= \|{\bf u} \|_{X} \}$ is   compact   and, hence,  the function ${\bf v}\mapsto  \langle {\bf v} , {\bf u}\rangle$ attains its maximum value at some  ${\bf v}^*\in B$.  Now, following Definition \ref{Def:DualConj}, one readily verifies that  $({\bf u},{\bf v}^*)$ is an  $({X},{X}')$-conjugate.

We conclude this part by providing a classical and illustrative example. Let $V= \mathbb{R}^n$ for some $n\in\mathbb{N}$. For any $p \in [1,+\infty]$, the $\ell_p$-norm of a vector ${\bf u} =(u_i)\in \mathbb{R}^n$ is defined as 
\begin{equation}\label{Eq:lp}
\|{\bf u}\|_{p} = \begin{cases} \left(\sum_{i=1}^n |u_i|^p\right)^{\frac{1}{p}}, & p< +\infty \\ \max_{i} |u_i|, & p= +\infty. \end{cases} 
\end{equation}
It is widely known that the dual norm of $\ell_p$ is the $\ell_q$-norm, where $(p,q)$ are H\"older conjugates ({\it i.e.},  $1/p+1/q=1$) \cite{rudin1991functional}. This stems from the H\"older inequality which states that 
\begin{equation}\label{Eq:HolderVec}
\langle {\bf v}, {\bf u} \rangle \leq \| {\bf u}\|_{p} \|{\bf v}\|_q,
\end{equation}
for all ${\bf u}=(u_i),{\bf v}=(v_i) \in \mathbb{R}^n$. In the sequel, we exclude the trivial cases   ${\bf u}=\boldsymbol{0}$ and ${\bf v}=\boldsymbol{0}$ to avoid unnecessary complexities in our statements.

When $1<p<+\infty$, Inequality \eqref{Eq:HolderVec} is saturated if and only if $u_i v_i \geq 0$ for $i=1,\ldots,n$ and  there exists a constant $c >0$ such that $|{\bf u}|^p = c |{\bf v}|^q$, where $|{\bf u}|^p = (|u_i|^p)$.  This ensures that the duality mapping is single-valued  and also yields the map 
\begin{equation}\label{Eq:DualMapLp}
\mathrm{J}_p({\bf u})= {\rm sign}({\bf u}) \frac{|{\bf u}|^{p-1}}{\|{\bf u}\|_p^{p-2}}. 
\end{equation}
 For $p=1$, one can verify that the equality happens if and only if,    for any index $i=1,\ldots,n$ with $u_i \neq 0$, one has that
\begin{equation}\label{Eq:HolderEq1}
v_i =   {\rm sign} (u_i)\|{\bf v}\|_{\infty}.
\end{equation}
In other words, the vector ${\bf v}$ should attain its extreme values at places where ${\bf u}$ has nonzero values, with the sign being determined by the corresponding element in ${\bf u}$. 

Due to \eqref{Eq:HolderEq1}, the set $\mathcal{J}_1({\bf u})$ is not necessarily a singleton. However,  if we add an additional sparsity constraint, then the mapping becomes single-valued. This leads us to introduce the new notion of {\it sparse duality mapping} in Definition \ref{Def:SparseDual}.
\begin{definition}\label{Def:SparseDual}
Let $V$ be a finite-dimensional vector space and let $s_0: V \rightarrow \mathbb{N}$ be an integer-valued function that acts as a sparsity measure.  Assuming a pair $(\|\cdot\|_X,\|\cdot\|_{X'})$ of dual norms over $V$, we call the pair $({\bf u},{\bf v})\in V\times V$  a sparse $(X,X')$-conjugate  if 
\begin{itemize}
\item $({\bf u},{\bf v})$ forms an $(X,X')$-conjugate pair. In other words, ${\bf v} \in \mathcal{J}({\bf u})$.
\item The quantity $s_0 ( {\bf v})$ attains its minimal value over the  set   $\mathcal{J}({\bf u})$.
\end{itemize}
We denote the set of sparse conjugates of ${\bf u}$ by $\mathcal{J}_{X,s_0}({\bf u})$. Whenever $\mathcal{J}_{X,s_0}({\bf u})$ is a singleton for any ${\bf u}\in V$, we refer to the single-valued function ${\rm J}_{X,s_0}: V \rightarrow V$ with $\mathcal{J}_{X,s_0}({\bf u})=\{{\rm J}_{X,s_0}({\bf u})\}$ as the sparse duality mapping. 
\end{definition}
Following Definition \ref{Def:SparseDual}, if we use the $\ell_0$-norm as the sparsity measure, that is $s_0 ({\bf u}) = \|{\bf u}\|_0={\rm Card}\left(\{i: u_i\neq 0\}\right)$\footnote{Although this functional  does not satisfy the homogeneity property of a norm, it has been widely referred to as the $\ell_0$-norm.}, then we have the sparse duality mapping 
\begin{align}\label{Eq:DualMapSparseL1}
&{\rm J}_{1,0}:\mathbb{R}^n\rightarrow\mathbb{R}^n:{\bf u}=(u_i) \mapsto {\bf v}=(v_i) ={\rm J}_{1,0}({\bf u}), \nonumber \\
&v_i = \begin{cases} {\rm sign}(u_i) \|{\bf u}\|_1, & u_i\neq 0 \\
0, & u_i=0.\end{cases} 
\end{align}
Finally, we mention that, for $p= +\infty$,  the  reduced set $\mathcal{J}_{\infty,0}$ is  not single-valued. Indeed, let us define $I_{\max}({\bf u})=\{i: |u_i|=\|{\bf u}\|_{\infty}\}\subseteq \{1,\ldots,n\}$. We readily deduce from \eqref{Eq:HolderEq1} that ${\bf v}=(v_1,\ldots,v_n) \in  \mathcal{J}_{\infty}({\bf u})$ if and only if $v_i = 0$ whenever $i\not \in I_{\max}({\bf u})$ and ${\rm 
sign}(v_i) = {\rm sign}(u_i)$ for $i \in I_{\max}({\bf u})$ with $\sum_{i\in I_{\max}({\bf u})} |
v_i|= \|{\bf u}\|_{\infty}$. This shows that $\mathcal{J}_{\infty}({\bf u})$ is a convex set with $\mathcal{J}_{\infty,0}({\bf u})$ being its extreme points, where  $\mathcal{J}_{\infty,0}({\bf u})=\{u_i {\bf e}_i: i\in  I_{\max}({\bf u})\}$.
\subsection{Schatten $p$-Norm}
It is widely known that any  matrix ${\bf A} \in \mathbb{R}^{m\times n}$  can be decomposed as 
\begin{equation}\label{Eq:SVD}
{\bf A}= {\bf U} {\bf S} {\bf V}^T,
\end{equation}
where ${\bf U}\in\mathbb{R}^{m\times m}$  and ${\bf V}\in\mathbb{R}^{n\times n}$ are orthogonal matrices and ${\bf S}$ is an $m$ by $n$ rectangular diagonal matrix with nonnegative real entries $\sigma_1 \geq \sigma_2 \geq \cdots \geq \sigma_{\min(m,n)} \geq 0$ sorted in descending order. In the literature, \eqref{Eq:SVD} is known as the singular-value decomposition (SVD)   and the entries $\sigma_i$  are the singular values of ${\bf A}$. In general, the SVD of a matrix $\bf A$ is not unique.  However, the diagonal matrix ${\bf S}$ and, consequently, its entries, are fully  determined  from ${\bf A}$. In other words, the values of $\sigma_i$ are invariant to a specific choice of decomposition. This is why one can refer to the diagonal entries of ${\bf S}$ as the ``singular values'' of ${\bf A}$.

 When ${\bf A}$ is not full rank, one can obtain a reduced version of \eqref{Eq:SVD}. Indeed, if we denote the rank of ${\bf A}$ by $r$, then we have that 
\begin{equation}\label{Eq:ReducedSVD}
{\bf A} = {\bf U}_r {\bf S}_r {\bf V}_r^T,
\end{equation}
where $ {\bf U}_r \in \mathbb{R}^{m\times r}$ and  $ {\bf V}_r \in \mathbb{R}^{n\times r}$   are (sub)-orthogonal matrices such that $ {\bf U}_r^T  {\bf U}_r =  {\bf V}_r^T {\bf V}_r =  {\bf I}_r$ and ${\bf S}_r ={\rm diag}(\boldsymbol{\sigma})$ is a diagonal matrix that contains positive singular values $\boldsymbol{\sigma}=(\sigma_1,\ldots,\sigma_r)\in \mathbb{R}^{r}$ of ${\bf A}$.  

Finally, for any  $p\in [1,+\infty]$, the Schatten-$p$ norm of ${\bf A}$ is defined as 
\begin{equation} \label{Eq:SchattenNorm} 
\|{\bf A}\|_{S_p} = \begin{cases} \left(\sum_{i=1}^r \sigma_i^p\right)^{\frac{1}{p}}, & p< +\infty \\ \sigma_1, & p= +\infty. \end{cases}  
\end{equation}
 
\section{Duality Mapping in Schatten Spaces}
\label{Sec:DualMap}
The dual of the Schatten-$p$ norm is the Schatten-$q$ norm, where $q\in [1,\infty]$ is such that  $\frac{1}{p}+\frac{1}{q}=1$ \cite{bhatia2013matrix}.  This is due to the generalized version of H\"older's inequality for Schatten norms, as stated in Proposition \ref{Prop:Holder}. While this is a known result (see, for example, \cite{Lefki2015StructureTensor}), it is also the basis for the present work, which is the reason why we provide a proof in  \ref{App:Holder}.
\begin{proposition}\label{Prop:Holder}
For any pair $(p,q)\in [1,+\infty]^2$ of H\"older conjugates with $\frac{1}{p}+\frac{1}{q}=1$  and any pair of  matrices ${\bf A},{\bf B}\in \mathbb{R}^{m\times n}$, we have that 
\begin{equation}\label{Eq:Holder}
\langle {\bf A}, {\bf B}\rangle = \mathrm{Tr}\left( {\bf A}^T {\bf B} \right) \leq \|{\bf A}\|_{S_p} \|{\bf B}\|_{S_q}. 
\end{equation}
\end{proposition}
In Proposition \ref{Prop:HolderSat}, we investigate the case where the H\"older inequality is saturated, in the sense that 
\begin{equation}\label{Eq:HolderSat}
{\rm Tr}\left( {\bf A}^T {\bf B} \right) = \|{\bf A}\|_{S_p} \|{\bf B}\|_{S_q}.
\end{equation}
This saturation is central to our work, as it is tightly linked to the notion of duality mapping.
\begin{proposition}\label{Prop:HolderSat}
Let $(p,q)$ be a pair of H\"older conjugates and let ${\bf A},{\bf B}\in \mathbb{R}^{m\times n}$ be a pair of nonzero matrices with reduced SVDs of the form 
\begin{equation}\label{Eq:RedSVD}
{\bf A}= {\bf U}_r {\rm diag}(\boldsymbol{\sigma}) {\bf V}_r^T,  \quad {\bf B}= \tilde{\bf U}_{\tilde{r}} {\rm diag}(\boldsymbol{\tilde{\sigma}})\tilde{\bf V}_{\tilde{r}}^T.
\end{equation}
\begin{itemize}
\item If $p\in (1,\infty)$, then the H\"older inequality is saturated if and only if we have that 
\begin{equation}\label{Eq:FormBgenP}
{\bf B} = c {\bf U}_r{\rm diag}({\rm J}_p(\boldsymbol{\sigma})){\bf V}_r^T
\end{equation}
or, equivalently, 
\begin{equation}\label{Eq:FormAgenP}
{\bf A} = c^{-1} {\bf \tilde{U}}_{\tilde{r}}{\rm diag}({\rm J}_q(\tilde{\boldsymbol{\sigma}})){\bf \tilde{V}}_{\tilde{r}}^T,
\end{equation}
where $c= \frac{\|{\bf B}\|_{S_q}}{\|{\bf A}\|_{S_p}}$ and ${\rm J}_p(\cdot)$ and ${\rm J}_q(\cdot)$ are the duality mappings for  the $\ell_p$ and $\ell_q$ norms, respectively (see \eqref{Eq:DualMapLp}). 
\item If $p=1$, then a necessary condition for the saturation of the H\"older inequality is that 
\begin{equation}
{\rm rank}({\bf A}) \leq r_1 \leq {\rm rank}({\bf B}),
\end{equation}
where $r_1= {\rm Card}\left(\{i: \tilde{\sigma}_i = \tilde{\sigma}_1\}\right)$ is  the multiplicity of the first singular value of ${\bf B}$. Moreover, if we denote the first $r_1$ singular vectors of ${\bf B}$ in \eqref{Eq:RedSVD} by ${\bf \tilde{U}}_1\in \mathbb{R}^{m\times r_1}$ and ${\bf \tilde{V}}_1\in\mathbb{R}^{n\times r_1}$, then the saturation of the H\"older inequality is equivalent to the existence of a symmetric matrix ${\bf X}\in \mathbb{R}^{r_1\times r_1}$ such that
\begin{align}\label{Eq:FormA1}
{\bf A}  = {\bf \tilde{U}}_1 {\bf X} {\bf \tilde{V}}_1^T.
\end{align}
Finally in the rank-equality case  ${\rm rank}({\bf A}) ={\rm rank}({\bf B})$, we have saturation if and only if 
\begin{equation}\label{Eq:FormBrankEq}
{\bf B} = c {\bf U}_r{\bf V}_r^T,
\end{equation}
where $c= \|{\bf B}\|_{S_{\infty}}$ and the matrices ${\bf U}_r$ and ${\bf V}_r$ are defined in \eqref{Eq:RedSVD}. 
\end{itemize}
\end{proposition}
\begin{remark}\label{Remark}
The reduced SVD is not unique; there are multiple choices for the sub-orthogonal  matrices in \eqref{Eq:RedSVD}. However, the parametric forms  given in Proposition \ref{Prop:HolderSat} do not depend on  a specific decomposition.
\end{remark}
 The proof of Proposition \ref{Prop:HolderSat} can be found in  \ref{App:HolderSat}. We observe that,  in the case $p\in(1,\infty)$, the saturation of H\"older inequality provides a very tight link between the two matrices: If we know one of them, then the other lies in a one-dimensional ray that is parameterized by the constant $c>0$. However, in the special case $p=1$, the identification is not as simple. There again, for a given  matrix ${\bf B}$,  one can fully characterize the set of admissible matrices ${\bf A}$. However, for the reverse direction,  an additional rank-equality  constraint  is essential to reduce the set of admissible matrices ${\bf B}$ to just   one ray. 

Inspired from Proposition \ref{Prop:HolderSat}, we now propose our main result  in Theorem \ref{Thm:p}, where  we explicitly characterize the duality mapping for the Schatten $p$-norms. The proof of Theorem \ref{Thm:p} can be found in  \ref{App:p}. 
\begin{theorem}\label{Thm:p}
Let $p,q\in [1,+\infty]$ be a pair of H\"older conjugates with $\frac{1}{p}+\frac{1}{q}=1$ and ${\bf A}\in\mathbb{R}^{m\times n}$ a matrix whose reduced SVD is specified in \eqref{Eq:ReducedSVD}.
\begin{itemize}
\item If $1<p<+\infty$, then the single-valued duality mapping ${\rm J}_{S_p}: \mathbb{R}^{m\times n} \rightarrow \mathbb{R}^{m\times n}$ is well-defined and can be expressed as  
\begin{equation}
{\rm J}_{S_p}:{\bf A} = {\bf U }_r {\rm diag}(\boldsymbol{\sigma} ) {\bf V}_r^T \mapsto {\bf A}^*= {\bf U}_r{\rm diag}({\rm J}_p(\boldsymbol{\sigma} )){\bf V}_r^T.
\end{equation}
\item If $p=1$ and if we consider the rank function as the sparsity measure in Definition \ref{Def:SparseDual}, then the sparse duality mapping ${\rm J}_{S_1,{\rm rank}}: \mathbb{R}^{m\times n} \rightarrow \mathbb{R}^{m\times n}$ is well-defined (singleton)  and is given as  
\begin{equation}
{\rm J}_{S_1,{\rm rank}}:{\bf A} = {\bf U }_r {\rm diag}(\boldsymbol{\sigma}) {\bf V}_r^T \mapsto {\bf A}^*= \|\boldsymbol{\sigma}\|_{1} {\bf U} _r{\bf V}_r^T.
\end{equation}
\item If $p=+\infty$, then the set-valued mapping $\mathcal{J}_{S_\infty}(\cdot)$  can be described as 
\begin{equation}\label{Eq:Jinf}
\mathcal{J}_{S_\infty}({\bf A})= \left\{  \sigma_1 {\bf U}_1 {\bf X} {\bf V}_1^T:  {\bf X}\in\mathbb{R}^{r_1\times r_1} \text{ is symmetric and } \|{\bf X}\|_{S_1} =1 \right\},
\end{equation}
where $r_1$ denotes the multiplicity of the first singular value $\sigma_1$ of ${\bf A}$ and ${\bf U}_1,{\bf V}_1$ are singular vectors that correspond to $\sigma_1$ in \eqref{Eq:ReducedSVD}. It is a convex set whose extreme points  are ${\bf E}_{i,j} = \frac{\sigma_1}{2}({\bf u}_i {\bf v}_j^T +{\bf v}_i {\bf u}_j^T )$ for $1\leq i\leq j\leq r_1$. Finally, the set of sparse dual conjugates is the collection of rank-1 elements of $\mathcal{J}_{S_\infty}({\bf A})$ which can be characterized as 
\begin{equation}
\mathcal{J}_{S_\infty, {\rm rank}}({\bf A}) = \{ \sigma_1 {\bf U}_1 {\bf p}{\bf p}^T {\bf V}_1^T: {\bf p}\in \mathbb{R}^{r_1} , \|{\bf p}\|_2=1\}.
\end{equation}
\end{itemize}
\end{theorem}

\section{Discussion}\label{Sec:discuss}
Theorem \ref{Thm:p} provides an interesting characterization of the duality mapping in three scenarios: The first case is $1<p<+\infty$ which is the most straightforward one. Theorem \ref{Thm:p} tells us  that the mapping is single-valued and also gives a formula to compute the dual conjugate ${\bf A}^*$ of any matrix ${\bf A}\in\mathbb{R}^{m\times n}$.  We use this result to deduce the continuity of the duality mapping as well as the strict convexity of the Schatten space in this case (see Corollary \ref{Cor:strict}).  In the second case, with $p=1$, the mapping is not single-valued. However, there is a unique element in the set of dual conjugates with the minimal rank (that is equal to the rank of ${\bf A}$) and, hence, we can construct a single-valued sparse duality mapping. Finally, we showed in the third case, characterized by  $p=+\infty$, that neither the set of dual conjugates nor the ones with the minimal rank are   unique. However, we observe in \eqref{Eq:Jinf} that the entries of ${\bf X}$ can be independently chosen, up to symmetry and normalization assumptions. This suggests that the  dimension of $\mathcal{J}_{S_\infty}({\bf A})$ is $d=\left( \frac{r_1(r_1+1)}{2}-1\right)$. Moreover, we show that this convex set has exactly $(d+1)$ extreme points, which is the minimal number for a convex set of dimension $d$. We also observe that the extreme points of $\mathcal{J}_{S_\infty}({\bf A})$ are low-rank. They are indeed a collection of rank-1 and rank-2 matrices. 

In Corollary \ref{Cor:strict}, we highlight some consequences of Theorem \ref{Thm:p} concerning the strict convexity of Schatten spaces and the continuity of the duality mapping. 
\begin{corollary}\label{Cor:strict}
The Banach space of $m$ by $n$ matrices equipped with the Schatten-$p$ norm is strictly convex, if and only if $p\in (1,+\infty)$. In this case,   the function ${\rm J}_{S_p} : \mathbb{R}^{m\times n} \rightarrow \mathbb{R}^{m\times n}$ is continuous.
\end{corollary}
\begin{proof}
For $p\in (1,+\infty)$, we know from Theorem \ref{Thm:p} that the duality mapping ${\rm J}_{S_p}$ is bijective. Moreover, it is known that all finite-dimensional Banach spaces are reflexive. Now, following \cite{petryshyn1970characterization}, we deduce the strict convexity of the space of $m$ by $n$ matrices with Schatten-$p$ norm.

For $p=1$ and $p=+\infty$, we can readily verify that
\begin{align*}
&\left \|\alpha \begin{pmatrix}1 & \cdots & 0 \\ \vdots & \ddots & \vdots \\ 0 & \cdots & 0 \end{pmatrix} +(1-\alpha)\begin{pmatrix}0 & \cdots & 0 \\ \vdots & \ddots & \vdots \\ 0 & \cdots & 1 \end{pmatrix}  \right\|_{S_1} = \left \|\begin{pmatrix} \alpha & \cdots & 0 \\ \vdots & \ddots & \vdots \\ 0 & \cdots & (1-\alpha) \end{pmatrix} \right\|_{S_1}=1, \\
&\left\|\alpha \begin{pmatrix}1 & \cdots & 0 \\ \vdots & \ddots & \vdots \\ 0 & \cdots & 0 \end{pmatrix}  +(1-\alpha)\begin{pmatrix}1 & \cdots & 0 \\ \vdots & \ddots & \vdots \\ 0 & \cdots & 1 \end{pmatrix}  \right\|_{S_\infty} = \left\| \begin{pmatrix}1 & \cdots & 0 \\ \vdots & \ddots & \vdots \\ 0 & \cdots & (1-\alpha) \end{pmatrix}  \right\|_{S_\infty}=1, 
\end{align*}
for all $\alpha \in (0,1)$, which  shows that the Schatten space is not strictly convex for $p=1,+\infty$. 

Finally, the Schatten-$p$ norm is  known to be  Fr\'echet differentiable  for $p\in(1,+\infty)$ \cite{potapov2014frechet}. Moreover,  the duality mapping of any  Banach space with Fr\'echet-differentiable norms is guaranteed to be continuous \cite{giles1978geometrical,contreras1994upper}. Combining the two statements, we deduce the continuity of the duality mapping in this case. 
\end{proof}
By contrast, the sparse duality mapping ${\rm J}_{S_1,{\rm rank}}(\cdot)$ is not continuous. This is best explained by providing a  counterexample. Specifically, let us consider the sequence of 2 by 2 matrices 
$${\bf S}_k=\begin{pmatrix} 1 & 0\\ 0 & \frac{1}{k}\end{pmatrix}, \quad k\in\mathbb{N}.$$
It is clear that ${\bf S}_k \rightarrow {\bf S}_{\infty} =\begin{pmatrix} 1 & 0\\ 0 & 0\end{pmatrix}$. However, we   have that 
$$\forall k\in\mathbb{N}: {\rm J}_{S_1,{\rm rank}}({\bf S}_k) = \begin{pmatrix} 1 & 0\\ 0 & 1\end{pmatrix}, \quad \text{while} \quad {\rm J}_{S_1,{\rm rank}}({\bf S}_{\infty}) = \begin{pmatrix} 1 & 0\\ 0 & 0\end{pmatrix},$$
which shows the discontinuity of ${\rm J}_{S_1,{\rm rank}}$ in the space of 2 by 2 matrices. This can be generalized to space of matrices with arbitrary dimensions $m,n\in\mathbb{N}$. 

Although ${\rm J}_{S_1,{\rm rank}}$ is not continuous, we  now show that it is Borel-measurable and, hence, that it can be approximated with arbitrary precision by a continuous mapping due to Lusin's theorem \cite{rudin1991functional}.
\begin{proposition}\label{Prop:measure}
For any $m,n\in\mathbb{N}$, the sparse duality mapping ${\rm J}_{S_1,{\rm rank}}$ is  a Borel-measurable matrix-valued function over the space of $m$ by $n$ matrices. 
\end{proposition} 
Before going into the proof of Proposition \ref{Prop:measure}, we present a preliminary result. 
\begin{lemma}\label{Lem}
The set  $\mathcal{R}_{ r}\subseteq \mathbb{R}^{m\times n}$  of $m$ by $n$  matrices of rank $r$ is Borel-measurable. 
\end{lemma}
\begin{proof}
First note that 
$$\mathcal{R}_{1}=\{ {\bf u}{\bf v}^T: {\bf u}\in\mathbb{R}^m, {\bf v}\in \mathbb{R}^n\}. $$
The set $\mathcal{R}_{1}$ is the image of the continuous mapping $\mathbb{R}^{m}\times \mathbb{R}^n \rightarrow \mathbb{R}^{m\times n}:({\bf u}, {\bf v}) \mapsto  {\bf u}{\bf v}^T$ and, hence, is Borel-measurable. 

Now, denote by  $\mathcal{R}_{\leq r}\subseteq \mathbb{R}^{m\times n}$, the set of matrices with rank no more than $r$. Using the identity 
$$\mathcal{R}_{\leq r} =  \mathcal{R}_{1} + \cdots + \mathcal{R}_{1} ,\quad \text{(r times)}, $$
we deduce that $\mathcal{R}_{\leq r}$ and, consequently, $\mathcal{R}_{r} =  \mathcal{R}_{\leq r}  \backslash \mathcal{R}_{\leq (r-1)}$ are also Borel-measurable sets.
\end{proof}
\begin{proof}[Proof of Proposition \ref{Prop:measure}]
Consider a Borel-measurable set $\mathcal{B}\subseteq \mathbb{R}^{m\times n}$. We show that $\mathcal{B}_{\rm inv}= {\rm J}_{S_1,{\rm rank}}^{-1}(\mathcal{B})$ is also Borel-measurable. By defining  $\mathcal{B}_{{\rm inv},r} = \mathcal{B}_{\rm inv}\cap \mathcal{R}_{r}$, we can partition   $ \mathcal{B}_{\rm inv}$ as
$$\mathcal{B}_{\rm inv}= \bigcup_{r=1}^{\min (m,n)}  \mathcal{B}_{\rm inv}\cap \mathcal{R}_{r}.$$
Hence, it is sufficient to show that each partition $\mathcal{B}_{\rm inv}\cap \mathcal{R}_{r}$ is Borel-measurable. 

Define the set $\mathcal{P}_{r} \subseteq \mathcal{R}_{r}^2$ as 
$$\mathcal{P}_{r} = \{ ({\bf A},{\bf B}) \in  \mathcal{R}_{r} \times \mathcal{B}: {\rm Tr}({\bf A}^T{\bf B}) = \|{\bf A}\|_{S_1} \|{\bf B}\|_{S_\infty}, \quad \|{\bf A}\|_{S_1}= \|{\bf B}\|_{S_\infty}\}.$$
The set $\mathcal{P}_{r}$ introduces a relation over $ \mathcal{R}_{r} $ whose domain is $\mathcal{B}_{\rm inv}\cap \mathcal{R}_{r}$. Since the trace and norm are continuous (and, consequently, Borel-measurable) functions and  $\mathcal{R}_{r}\times \mathcal{B}$ is a Borel-measurable set (using Lemma \ref{Lem}), we deduce that the relation induced from $\mathcal{P}_{r}$ is Borel-measurable as well. Finally, we use \cite[Proposition 2.1]{himmelberg1975measurable} to show that its domain is Borel-measurable.  
\end{proof}

\section{Conclusion}
In this paper, we studied the duality mapping in finite-dimensional Schatten spaces. Based on a careful investigation of the cases where the  H\"older inequality saturates, we provided an explicit form for this mapping when $p\in (1,+\infty)$. Furthermore,  by adding a rank constraint, we proved that the mapping becomes single-valued for the special case $p=1$. As for $p=+\infty$, we showed that the mapping yields a convex set whose extreme points are low-rank matrices. Finally, we discussed our theorem and studied the continuity of the introduced mappings as well as the strict convexity of the Schatten spaces. A possible future direction of research is to extend the results of this paper to infinite-dimensional Schatten spaces  and even, in full generality, to linear operators over Hilbert spaces.

\appendix
\section{Proof of Proposition \ref{Prop:Holder}}\label{App:Holder}
\begin{proof}
Let us recall the reduced SVD of the matrix ${\bf A}$ as 
\begin{equation}\label{Eq:RedSVDA}
{\bf A}= {\bf U}_r {\bf S}_r {\bf V}_r^T,  
\end{equation}
where $r={\rm rank}({\bf A})$, ${\bf U}_r = [{\bf u}_1 \cdots {\bf u}_r]\in \mathbb{R}^{m\times r}$, ${\bf V}_r = [{\bf v}_1 \cdots {\bf v}_r]\in \mathbb{R}^{n\times r}$, and ${\bf S}= {\rm diag}(\sigma_1,\ldots,\sigma_r)$. Similarly, for the matrix ${\bf B}$, we have that  
\begin{equation}\label{Eq:RedSVDB}
{\bf B}= \tilde{\bf U}_{\tilde{r}} \tilde{\bf S}_{\tilde{r}}\tilde{\bf V}_{\tilde{r}}^T,
\end{equation}
where $\tilde{r}={\rm rank}({\bf A})$, $\tilde{\bf U}_{\tilde{r}} = [\tilde{\bf u}_1 \cdots \tilde{\bf u}_{\tilde{r}}]\in \mathbb{R}^{m
\times \tilde{r}}$, $\tilde{\bf V}_{\tilde{r}} = [\tilde{\bf v}_1 \cdots \tilde{\bf v}_r]\in \mathbb{R}
^{n\times \tilde{r}}$, and $\tilde{\bf S}= {\rm diag}(\tilde{\sigma}_1,\ldots,\tilde{\sigma}
_{\tilde{r}})$. A direct computation then reveals that 
\begin{equation}\label{Eq:TraceExpand}
{\rm Tr}\left({\bf A}^T {\bf B}\right) = \sum_{i=1}^r \sum_{j=1}^{\tilde{r}} \sigma_i \tilde{\sigma}_j {\bf u}_i^T  \tilde{\bf u}_j {\bf v}_i^T \tilde{\bf v}_j. 
\end{equation}
By using the weighted H\"older inequality for vectors \cite{cvetkovski2012inequalities}, we obtain for $p\neq 1$ that 
\begin{equation}\label{Eq:WeightedHolderp}
 \sum_{i=1}^r \sum_{j=1}^{\tilde{r}} \sigma_i \tilde{\sigma}_j {\bf u}_i^T \tilde{\bf u}_j {\bf v}_i ^T \tilde{\bf v}_j\leq \left(\sum_{i=1}^r\sigma_i^p \sum_{j=1}^{\tilde{r}} |{\bf u}_i^T  \tilde{\bf u}_j {\bf v}_i^T \tilde{\bf v}_j| \right)^{\frac{1}{p}}\left(\sum_{j=1}^{\tilde{r}}\sigma_j^p \sum_{i=1}^{r} |{\bf u}_i^T  \tilde{\bf u}_j{\bf v}_i^T \tilde{\bf v}_j| \right)^{\frac{1}{q}}
\end{equation}
and for $p=1$ that
\begin{equation}\label{Eq:WeightedHolder1}
 \sum_{i=1}^r \sum_{j=1}^{\tilde{r}} \sigma_i \tilde{\sigma}_j {\bf u}_i^T  \tilde{\bf u}
 _j{\bf v}_i^T \tilde{\bf v}_j\leq  \left(\sum_{i=1}^r \sigma_i \sum_{j=1}^{\tilde{r}}  |{\bf u}_i^T 
 \tilde{\bf u}_j{\bf v}_i^T \tilde{\bf v}_j| \right) \|\tilde{\boldsymbol{\sigma}}\|_{\infty}.
\end{equation}
Finally, by invoking Cauchy-Schwartz and the orthonormality of the matrices ${\bf U}_r,{\bf V}_r, \tilde{\bf U}_{\tilde{r}}, \tilde{\bf V}_{\tilde{r}}$, we deduce for $i=1,\ldots,r$ that 
\begin{equation}\label{Eq:CauchyJ}
\sum_{j=1}^{\tilde{r}} |{\bf u}_i^T \tilde{\bf u}_j {\bf v}_i^T \tilde{\bf v}_j| \leq \left( \sum_{j=1}^{\tilde{r}} ({\bf u}_i^T \tilde{\bf u}_j)^2\right)^{\frac{1}{2}} \left( \sum_{j=1}^{\tilde{r}} ({\bf v}_i^T \tilde{\bf v}_j)^2\right)^{\frac{1}{2}}  \leq \|{\bf u}_i\|_2 \|{\bf v}_i\|_2 = 1,
\end{equation}
For $j=1,\ldots, \tilde{r}$, we deduce  that 
\begin{equation}\label{Eq:CauchyI}
\sum_{i=1}^{r} |{\bf u}_i^T \tilde{\bf u}_j {\bf v}_i^T \tilde{\bf v}_j| \leq \left( \sum_{i=1}^{r} ({\bf u}_i^T \tilde{\bf u}_j)^2\right)^{\frac{1}{2}} \left( \sum_{i=1}^{r} ({\bf v}_i^T \tilde{\bf v}_j)^2\right)^{\frac{1}{2}}  \leq \|\tilde{\bf u}_i\|_2 \|\tilde{\bf v}_i\|_2 = 1.
\end{equation}
The combination of these inequalities completes the proof. 
\end{proof}

\section{Proof of Proposition \ref{Prop:HolderSat}}\label{App:HolderSat}
\begin{proof}
We separate the two cases and analyze each one  independently. 

{\bf Case 1: $1<p<+\infty$}. We prove \eqref{Eq:FormBgenP} and deduce \eqref{Eq:FormAgenP} by symmetry.  Following the proof of Proposition \ref{Prop:HolderSat} and considering the reduced SVD of  the matrices ${\bf A}$ and ${\bf B}$  given in \eqref{Eq:RedSVDA} and \eqref{Eq:RedSVDB}, we immediately see that  the inequalities \eqref{Eq:WeightedHolderp}, \eqref{Eq:CauchyJ}, and \eqref{Eq:CauchyI} should all be saturated. The equality condition of the weighted H\"older implies the existence of a positive constant $\alpha>0$ such that, for all $(i,j)\in \{1,\ldots,r \}\times \{1,\ldots, \tilde{r}\}$, we have one of the following conditions: 
\begin{align}
&{\bf u}_i^T  \tilde{\bf u}_j{\bf v}_i^T \tilde{\bf v}_j = 0, \quad \text{or}\label{Eq:Zero}\\
&{\bf u}_i^T  \tilde{\bf u}_j{\bf v}_i^T \tilde{\bf v}_j > 0 \quad \text{and} \quad\tilde{\sigma}_j^q= \alpha\sigma_i^p  \label{Eq:LinDep}.
\end{align}
Moreover, the saturation of \eqref{Eq:CauchyJ} implies that 
\begin{equation}\label{Eq:UinRangeUtield}
{\bf u}_i \in {\rm Range}(\tilde{\bf U}_{\tilde{r}}), \quad {\bf v}_i \in {\rm Range}(\tilde{\bf V}_{\tilde{r}})  \qquad \forall i=1,\ldots, r 
\end{equation}
and also that  there exists a positive constant $\beta_i>0$ (positivity follows from \eqref{Eq:LinDep} and \eqref{Eq:Zero})  such that 
\begin{equation}\label{Eq:CauchySatI}
{\bf u}_i^T \tilde{\bf u}_j = \beta_i  {\bf v}_i^T \tilde{\bf v}_j , \quad \forall j=1,\ldots, \tilde{r}.
\end{equation}
However, from the normality of ${\bf u}_i$ and \eqref{Eq:UinRangeUtield}, we have that 
\begin{equation}
1= \|{\bf u}_i\|_2^2 = \sum_{j=1}^{\tilde{r}} |{\bf u}_i^T \tilde{\bf u}_j |^2 = \beta_i^2 \sum_{j=1}^{\tilde{r}} |{\bf v}_i^T \tilde{\bf v}_j |^2 = \beta_i^2 \|{\bf v}\|_2^2  = \beta_i^2
\end{equation}
which, together with the positivity of $\beta_i$, leads to the conclusion that  $\beta_i=1$ for $i=1,\ldots,r$. Using this, we rewrite \eqref{Eq:CauchySatI} in matrix form as 
\begin{equation}\label{Eq:UeqV}
{\bf U}^T_r \tilde{\bf U}_{\tilde{r}} = {\bf V}^T_r \tilde{\bf V}_{\tilde{r}}.
\end{equation}

Similarly, the saturation of \eqref{Eq:CauchyI} implies that 
\begin{equation}\label{Eq:UtieldinRangeU}
\tilde{\bf u}_j \in {\rm Range}({\bf U}_r), \quad \tilde{\bf v}_i \in {\rm Range}({\bf V}_r), 
\end{equation}
for all $j=1,\ldots,\tilde{r}$. Putting together \eqref{Eq:UinRangeUtield} and \eqref{Eq:UtieldinRangeU}, we deduce that $r=\tilde{r}$ and
\begin{equation}\label{Eq:RanEq}
{\rm Range}({\bf U}_r) = {\rm Range}(\tilde{\bf U}_{\tilde{r}}), \quad {\rm Range}({\bf V}_r) = {\rm Range}(\tilde{\bf V}_{\tilde{r}}).
\end{equation}
This implies the existence of two orthogonal matrices ${\bf P},{\bf Q}\in\mathbb{R}^{r\times r}$ such that 
\begin{equation}\label{Eq:UtieldRep}
\tilde{\bf U}_{\tilde{r}} = {\bf U}_r {\bf P}, \quad \tilde{\bf V}_{\tilde{r}} = {\bf V}_r {\bf Q}.
\end{equation}
However,  replacing \eqref{Eq:UtieldRep} in \eqref{Eq:UeqV}, we conclude that 
\begin{equation}\label{Eq:PeqQ}
{\bf P} =  {\bf U}^T_r {\bf U}_r {\bf P}={\bf U}^T_r \tilde{\bf U}_{\tilde{r}} ={\bf V}^T_r \tilde{\bf V}_{\tilde{r}}=   {\bf V}^T_r {\bf V}_r {\bf Q} ={\bf Q}.
\end{equation} 
This implies that the matrix ${\bf B}$ can be represented as 
\begin{equation}\label{Eq:Brep}
{\bf B}= {\bf U}_r {\bf P} \tilde{\bf S}_{\tilde{r}} {\bf P}^T {\bf V}^T_r = {\bf U}_r {\bf S}_0 {\bf V}_r^T,
\end{equation}
where ${\bf S}_0 = {\bf P} \tilde{\bf S}_{\tilde{r}} {\bf P}^T$. We now show that ${\bf S}_0$ is a diagonal matrix. Indeed, by denoting the $(i,j)$-th entry of ${\bf P}$ as $p_{i,j}$ such that ${\bf P} = [{\bf p}_1 \cdots {\bf p}_r] = [p_{i,j}]$, we   rewrite  \eqref{Eq:Zero} and \eqref{Eq:LinDep} as 
\begin{align}
&p_{i,j}= 0, \quad \text{or}\label{Eq:Zero2}\\
&p_{i,j}>0 \quad \text{and} \quad \tilde{\sigma}_j^q= \alpha\sigma_i^p   \label{Eq:LinDep2},
\end{align}
for all $(i,j)\in \{1,\ldots,r\}^2$. Moreover, by expanding the $(i,j)$-th entry of the matrix ${\bf S}_0$, we  have that 
\begin{align*}
[{\bf S}_0]_{i,j}  = [{\bf P} \tilde{\bf S}_{\tilde{r}} {\bf P}^T]_{i,j} &= \sum_{k=1}^r p_{i,k} \tilde{\sigma}_k p_{j,k}= \sum_{k=1}^r p_{i,k}  {\sigma}_i^{\frac{p}{q}} \alpha^{\frac{1}{q}} p_{j,k} \\& = {\sigma}_i^{\frac{p}{q}} \alpha^{\frac{1}{q}} {\bf p}_i^T {\bf p}_j  = [{\rm J}_p (\boldsymbol{\sigma})]_i c_{\bf B} \delta[i-j],
\end{align*}
where $\delta[\cdot]$ denotes the Kronecker delta and  $c_{\bf B}=  \alpha^{\frac{1}{q}} >0$ is a positive constant. Finally, we obtain the announced expression in \eqref{Eq:FormBgenP} by replacing the above characterization of ${\bf S}_0$ in \eqref{Eq:Brep}.  

For the converse, we note that, if the matrix $\bf B$ is in the form of \eqref{Eq:FormBgenP}, then we have that 
\begin{align*}
{\rm Tr}\left({\bf A}^T {\bf B}\right) &= {\rm Tr}\left( {\bf U}_r{\rm diag}( \boldsymbol{\sigma}){\bf V}_r^T  \left({\bf U}_r{\rm diag}({\rm J}_p(\boldsymbol{\sigma})){\bf V}_r^T\right)^T \right) \\
&= c_{\bf B} {\rm Tr}\left( {\rm diag}( \boldsymbol{\sigma}){\bf V}_r^T  {\bf V}_r {\rm diag}({\rm J}_p(\boldsymbol{\sigma})) {\bf U}_r^T {\bf U}_r\right) 
\\& = c_{\bf B}  \boldsymbol{\sigma}^T {\rm J}_p(\boldsymbol{\sigma}) 
\\&=  c_{\bf B} \|\boldsymbol{\sigma}\|_p \|{\rm J}_p(\boldsymbol{\sigma})\|_q  = \|{\bf A}\|_{S_p} \|{\bf B}\|_{S_q},
\end{align*}
which shows that the equality is indeed saturated in this case. 

{\bf Case 2: $p=1$}. In this case, the saturation of the weighted H\"older inequality implies that,  for all $(i,j)\in \{1,\ldots, r\}\times \{1,\ldots, \tilde{r}\}$, we have that 
\begin{align}
&{\bf u}_i^T  \tilde{\bf u}_j{\bf v}_i^T \tilde{\bf v}_j = 0, \quad \text{or}\label{Eq:Zero1}\\
&{\bf u}_i^T  \tilde{\bf u}_j{\bf v}_i^T \tilde{\bf v}_j > 0 \quad \text{and} \quad\tilde{\sigma}_j=\tilde{\sigma}_1\label{Eq:LinDep1}.
\end{align}
For  equality, we also need to have the saturation of \eqref{Eq:CauchyJ}, which we showed to be equivalent to \eqref{Eq:UinRangeUtield} and \eqref{Eq:UeqV}. From \eqref{Eq:UinRangeUtield}, we deduce   the existence of matrices ${\bf P}_1, {\bf P}_2\in\mathbb{R}^{\tilde{r}\times r}$ such that 
\begin{equation}\label{Eq:UUtieldP}
{\bf U}_r = {\bf \tilde{U}}_{\tilde{r}} {\bf P}_1, \qquad {\bf V}_r = {\bf \tilde{V}}_{\tilde{r}} {\bf P}_2.
\end{equation}
The replacement of these in \eqref{Eq:UeqV} implies that 
\begin{equation}
{\bf P}_1^T = {\bf P}_1^T  {\bf \tilde{U}}^T_{\tilde{r}} {\bf \tilde{U}}_{\tilde{r}} = {\bf {U}}^T_r {\bf \tilde{U}}_{\tilde{r}} ={\bf {V}}^T_r {\bf \tilde{V}}_{\tilde{r}} = {\bf P}_2^T  {\bf \tilde{V}}^T_{\tilde{r}} {\bf \tilde{V}}_{\tilde{r}} = {\bf P}_2^T,
\end{equation}
and, hence, that ${\bf P}_1={\bf P}_2=  [p_{i,j}]\in\mathbb{R}^{\tilde{r}\times r}$. Now, one can rewrite the conditions \eqref{Eq:Zero1} and \eqref{Eq:LinDep1} and deduce that, for any $j=1,\ldots, \tilde{r}$, we have that 
\begin{align}\label{Eq:LinDep1P}
&p_{j,i}= 0, \quad \forall i=1,\ldots,r \quad \text{or} \quad \tilde{\sigma}_j=\tilde{\sigma}_1. 
\end{align}
From Conditions \eqref{Eq:LinDep1P} and following the definition of $r_1$ (the multiplicity of the largest singular value), we deduce that 
\begin{equation}
{\bf P}_1 = \begin{bmatrix}
{\bf P} \\ \boldsymbol{0}_{r_{\rm res} \times r}
\end{bmatrix}, 
\end{equation}
where ${\bf P}\in \mathbb{R}^{r_1 \times r}$ and  $r_{\rm res} = (\tilde{r} -r_1)$.  Using this form and the definition of   ${\bf \tilde{U}}_1$ and ${\bf \tilde{V}}_1$ (given in the statement of the proposition), we  rewrite \eqref{Eq:UUtieldP} as 
\begin{equation}\label{Eq:UUtieldP2}
{\bf U}_r = {\bf \tilde{U}}_{1} {\bf P}, \qquad {\bf V}_r = {\bf \tilde{V}}_{1} {\bf P}.
\end{equation}
Therefore, 
\begin{equation} 
{\bf I}_r = {\bf U}_r^T {\bf U}_r = {\bf P}^T {\bf \tilde{U}}_{1}^T {\bf \tilde{U}}_{1} {\bf P}=  {\bf P}^T {\bf P}.
\end{equation}
Hence, ${\bf P}$ is a sub-orthogonal matrix and 
$${\rm rank}({\bf B}) = \tilde{r} \geq r_1 \geq {\rm rank}({\bf P}) \geq r={\rm rank}({\bf A}).$$
The replacement of \eqref{Eq:UUtieldP2} in the reduced SVD of ${\bf A}$ yields the announced expression with ${\bf X}= {\bf P} {\bf S} {\bf P}^T$.

 Based on the definitions of $r_1$, ${\bf \tilde{U}}_1$, and ${\bf \tilde{V}}_1$, we note that one can rewrite the reduced SVD of ${\bf B}$ as 
\begin{equation}\label{Eq:RedSVDB2}
{\bf B} =  \tilde{\sigma}_1 {\bf \tilde{U}}_1 {\bf \tilde{V}}_1^T  + {\bf \tilde{U}}_{\rm res} {\bf \tilde{S}}_{\rm res} {\bf \tilde{V}}_{\rm res}^T, 
\end{equation}
where ${\bf \tilde{U}}_{\rm res}\in\mathbb{R}^{m\times r_{\rm res}}$,  ${\bf \tilde{S}}_{\rm res}\in\mathbb{R}^{r_{\rm res}\times r_{\rm res}}$, and ${\bf \tilde{V}}_{\rm res}\in \mathbb{R}^{n\times r_{\rm res}}$ are the remaining singular values and vectors  such that  
$${\bf \tilde{U}} = [{\bf \tilde{U}}_1 \quad {\bf \tilde{U}}_{\rm res}], \quad  {\bf \tilde{V}} = [{\bf \tilde{V}}_1 \quad {\bf \tilde{V}}_{\rm res}], \quad  {\bf \tilde{S}}= \begin{bmatrix} \tilde{\sigma}_1 {\bf I}_{r_1} & \boldsymbol{0} \\ \boldsymbol{0}  & \tilde{S}_{\rm res}\end{bmatrix}.$$ 
Now, if ${\bf A}$ admits the form \eqref{Eq:FormA1} and if we consider the SVD of ${\bf X}={\bf P} {\bf S}{\bf P}^T$ (the assumption that ${\bf X}$ is symmetric ensures that is has an orthogonal eigen-decomposition), then 
\begin{align*}
{\rm Tr}\left({\bf A}^T {\bf B}\right) &=  {\rm Tr}\left( {\bf \tilde{V}}_1 {\bf P} {\bf S} {\bf P}^T {\bf \tilde{U}}_1 ^T \left(  \tilde{\sigma}_1 {\bf \tilde{U}}_1 {\bf \tilde{V}}_1^T  + {\bf \tilde{U}}_{\rm res} {\bf \tilde{S}}_{\rm res} {\bf \tilde{V}}_{\rm res}^T\right) \right) \\
& = \tilde{\sigma}_1 {\rm Tr}\left( {\bf \tilde{V}}_1 {\bf P} {\bf S} {\bf P}^T {\bf \tilde{U}}_1 ^T{\bf \tilde{U}}_1 {\bf \tilde{V}}_1^T\right) + {\rm Tr}\left( {\bf \tilde{V}}_1 {\bf P} {\bf S} {\bf P}^T {\bf \tilde{U}}_1 ^T {\bf \tilde{U}}_{\rm res} {\bf \tilde{S}}_{\rm res} {\bf \tilde{V}}_{\rm res}^T\right)
\\&= \tilde{\sigma}_1 {\rm Tr}\left( {\bf \tilde{V}}_1 {\bf P} {\bf S} {\bf P}^T {\bf I}_{r_1} {\bf \tilde{V}}_1^T\right)+ {\rm Tr}\left( {\bf \tilde{V}}_1 {\bf P} {\bf S} {\bf P}^T \boldsymbol{0}_{r_1\times r_{\rm res}} {\bf \tilde{S}}_{\rm res} {\bf \tilde{V}}_{\rm res}^T\right)
\\& = \tilde{\sigma}_1 {\rm Tr}\left( {\bf S} {\bf P}^T  {\bf \tilde{V}}_1^T{\bf \tilde{V}}_1 {\bf P} \right)+ 0 
\\& = \tilde{\sigma}_1 {\rm Tr}\left( {\bf S} {\bf P}^T   {\bf P} \right)
\\& = \tilde{\sigma}_1 {\rm Tr}\left( {\bf S}\right) = \|{\bf B}\|_{S_\infty} \|{\bf A}\|_{S_1},
\end{align*}
which establishes the  sufficiency in this case.

Finally, assuming that $r=r_1= \tilde{r}$, we deduce that ${\bf P}\in\mathbb{R}^{r\times r}$ is an orthogonal matrix and, hence, that  ${\bf P}^{-1}= {\bf P}^T$. Now, using \eqref{Eq:UUtieldP2} and the rank assumption, we can simplify the expansion \eqref{Eq:RedSVDB2} as 
\begin{equation} 
{\bf B} =  \tilde{\sigma}_1 {\bf \tilde{U}}_1 {\bf \tilde{V}}_1^T  = \tilde{\sigma}_1 {\bf U}_r {\bf P}^T  \left({\bf V}_r {\bf P}^T\right)^T =  \tilde{\sigma}_1 {\bf U}_r {\bf P}^T  {\bf P} {\bf V}_r^T =  \tilde{\sigma}_1 {\bf U}_r  {\bf V}_r^T.
\end{equation}
 \end{proof}
 \section{Proof of Theorem \ref{Thm:p}}\label{App:p}
 \begin{proof} {\bf Case I: $1<p<+\infty$.} Assume that $({\bf A},{\bf B})$ forms an $(S_p,S_q)$-conjugate pair. Hence, we  have that  $\langle {\bf A}, {\bf B}\rangle = \|{\bf A}\|_{S_p} \|{\bf B}\|_{S_q}$ which, together with  Proposition \ref{Prop:HolderSat},  implies  that ${\bf B}$   admits the form 
$$ {\bf B}= \frac{\|{\bf B}\|_{S_q}}{\|{\bf A}\|_{S_p}} {\bf U}_r{\rm diag}({\rm J}_p(\boldsymbol{\sigma})){\bf V}_r^T={\bf U}_r{\rm diag}({\rm J}_p(\boldsymbol{\sigma})){\bf V}_r^T.$$

{\bf Case II: $p=1$.} Similarly to the previous case, consider ${\bf A} \in \mathbb{R}^{m\times n}$ and ${\bf B}\in\mathcal{J}_{S_1,{\rm rank}}({\bf A})$. We have that 
\begin{align}
&{\rm Tr}\left({\bf A}^T{\bf B}\right) = \|{\bf A}\|_{S_1}\|{\bf B}\|_{S_{\infty}} \label{Eq:C1} \\
&\|{\bf A}\|_{S_1} = \|{\bf B}\|_{S_{\infty}} \label{Eq:C2} \\
& {\rm rank}({\bf B}) \leq {\rm rank}({\bf C}), \quad \forall {\bf C} \in \mathcal{J}_{S_1}({\bf A}). \label{Eq:C3}
\end{align} 
From \eqref{Eq:C1} and using Proposition \ref{Prop:HolderSat}, we deduce that ${\rm rank}({\bf B})\geq {\rm rank}({\bf A})$ which, together with \eqref{Eq:C3}, implies that ${\bf B}$ should be equal to 
$$ {\bf B} = \|{\bf B}\|_{S_{\infty}} {\bf U}_r {\bf V}_r^T= \|\boldsymbol{\sigma}\|_1 {\bf U}_r {\bf V}_r^T,$$
where the last equality is obtained using \eqref{Eq:C2}.

{\bf Case III: $p=+\infty$.} Following Proposition \ref{Prop:HolderSat},  any matrix ${\bf B} \in \mathcal{J}_{S_\infty}({\bf A})$ can be expressed as 
$${\bf B} = {\bf U}_1 {\bf \tilde{X}} {\bf V}_1^T, $$
where ${\bf \tilde{X}}\in\mathbb{R}^{r_1\times r_1}$ is a symmetric matrix. By defining ${\bf X} = \sigma_1^{-1} {\bf \tilde{X}}$, one readily verifies that  ${\bf B} =\sigma_1 {\bf U}_1 {\bf {X}} {\bf V}_1^T $. By recalling the normalization constraint $\|{\bf A}\|_{S_\infty} = \|{\bf B}\|_{S_1}$, we therefore obtain that 
$$ \sigma_1 = \|{\bf A}\|_{S_\infty} = \|{\bf B}\|_{S_1} = \sigma_1 \|{\bf X}\|_{S_1},$$ 
which implies that $\|{\bf X}\|_{S_1}=1$. To show that $\mathcal{J}_{S_\infty}({\bf A})$ is convex, consider two symmetric matrices ${\bf X}_0$ and ${\bf X}_1$ in the unit ball of Schatten-1 norm and define 
$$ {\bf B}_{\alpha} = \sigma_1 {\bf U}_1 {\bf {X}}_\alpha {\bf V}_1^T, \quad {\bf X}_{\alpha} = \alpha {\bf X}_0 + (1-\alpha) {\bf X}_1 $$
 for $\alpha \in [0,1]$. On one hand, from the linearity of traces, we have  that
$$ {\rm Tr}({\bf A}^T{\bf B}_{\alpha} ) = {\rm Tr}\left({\bf A}^T \left(\alpha {\bf B}_1 + (1-\alpha){\bf B}_0\right) \right) = \alpha {\rm Tr}({\bf A}^T {\bf B}_1) + (1-\alpha){\rm Tr}({\bf A}^T {\bf B}_0).$$
On the other hand, from the definition of ${\bf X}_0$ and ${\bf X}_1$, we deduce that ${\bf B}_0,{\bf B}_1 \in \mathcal{J}_{S_\infty}({\bf A})$. Hence, 
$$ {\rm Tr}({\bf A}^T{\bf B}_{\alpha} ) =  \alpha \|{\bf A}\|_{S_\infty}^2 + (1-\alpha)\|{\bf A}\|_{S_\infty}^2  = \|{\bf A}\|_{S_\infty}^2.$$
However,  from the H\"older inequality and the convexity of norms, we have that 
$$ {\rm Tr}({\bf A}^T{\bf B}_{\alpha} ) \leq \|{\bf A}\|_{S_\infty} \|{\bf B}_{\alpha}\|_{S_1} \leq  \|{\bf A}\|_{S_\infty} \left(\alpha \|{\bf B}_1\|_{S_1} + (1- \alpha ) \|{\bf B}_0\|_{S_1}\right) = \|{\bf A}\|_{S_{\infty}}.$$
This implies that the H\"older inequality is saturated and also that  $\|{\bf B}_{\alpha}\|_{S_1} = \|{\bf A}\|_{S_\infty}$ which, altogether, implies that ${\bf B}_{\alpha} \in \mathcal{J}_{S_\infty}({\bf A})$ for all $\alpha\in [0,1]$. 

It is clear that $ \mathcal{J}_{S_\infty}({\bf A})$ is  the convex hull of ${\bf E}_{i,j}$'s for $1\leq i\leq j\leq r_1$ and that the ${\bf E}_{i,j}$ are linearly independent. This proves that  ${\bf E}_{i,j}$s are indeed the extreme points of $  \mathcal{J}_{S_\infty}({\bf A})$. 

Finally, we observe that the set $ \mathcal{J}_{S_\infty}({\bf A})$ contains all matrices of the  form ${\bf B} = {\bf U}_1 {\bf p}{\bf p}^T{\bf V}_1^T$ for any vector ${\bf p}\in\mathbb{R}^{r_1}$ with $\|{\bf p}\|_2=1$. These are indeed  all the rank-1 elements of $ \mathcal{J}_{S_\infty}({\bf A})$ which, due to the Definition \ref{Def:SparseDual}, forms the set of sparse dual conjugates. 
\end{proof}

\bibliography{Aziznejad.bib}

\end{document}